\documentclass[a4paper,12pt]{amsart}
\usepackage[isbn=false,sorting=nyt,backend=bibtex,maxnames=5,maxalphanames=5,firstinits=true, style=alphabetic]{biblatex}
\bibliography{main}

\PassOptionsToPackage{hyphens}{url}
\usepackage[colorlinks,plainpages,hypertexnames=false,plainpages=false]{hyperref}
\hypersetup{urlcolor=blue, citecolor=blue, linkcolor=blue}

\usepackage{comment}
\usepackage{amsmath}
\usepackage{mathtools}
\usepackage{tikz-cd}
\usepackage{bm}
\usepackage{amsthm}
\usepackage{amssymb}
\usepackage{amsfonts}
\usepackage{enumerate}
\usepackage{stmaryrd}
\usepackage{enumitem}
\usepackage{listings}
\usepackage[noend]{algorithmic} % for the algorithmic enviroment

\usepackage{array}
\newcolumntype{L}[1]{>{\raggedright\let\newline\\\arraybackslash\hspace{0pt}}m{#1}}
\newcolumntype{C}[1]{>{\centering\let\newline\\\arraybackslash\hspace{0pt}}m{#1}}
\newcolumntype{R}[1]{>{\raggedleft\let\newline\\\arraybackslash\hspace{0pt}}m{#1}}

\usepackage{tikz}
\usetikzlibrary{arrows}
\usetikzlibrary{decorations.pathreplacing}
\usetikzlibrary{matrix}
\usetikzlibrary{calc}
\usetikzlibrary{shapes}
\usetikzlibrary{patterns}
\usetikzlibrary{fit,backgrounds,scopes}

\newtheoremstyle{theoremstyle}
{10pt}      %  Space above
{5pt}       %  Space below
{\itshape}  %  Body font
{}          %  Indent amount (empty = no indent, \parindent = para indent)
{\bfseries} %  Thm head font
{}         %  Punctuation after thm head
{ }      %  Space after thm head: " " = normal interword space;
{}          %  Thm head spec (can be left empty, meaning `normal')

\newtheoremstyle{algorithmstyle}
{10pt}      %  Space above
{5pt}       %  Space below
{}  %  Body font
{}          %  Indent amount (empty = no indent, \parindent = para indent)
{\bfseries} %  Thm head font
{}         %  Punctuation after thm head
{ }      %  Space after thm head: " " = normal interword space;
{}          %  Thm head spec (can be left empty, meaning `normal')

\newtheoremstyle{examplestyle}
{10pt}      %  Space above
{5pt}       %  Space below
{}          %  Body font
{}          %  Indent amount (empty = no indent, \parindent = para indent)
{\bfseries} %  Thm head font
{}         %  Punctuation after thm head
{ }      %  Space after thm head: " " = normal interword space;
{}          %  Thm head spec (can be left empty, meaning `normal')

\theoremstyle{theoremstyle}
\newtheorem{theorem}{Theorem}[section]
\newtheorem*{theorem*}{Theorem}
\newtheorem{lemma}[theorem]{Lemma}
\newtheorem{proposition}[theorem]{Proposition}
\newtheorem*{proposition*}{Proposition}

\newtheorem*{corollary*}{Corollary}
\newtheorem{conjecture}[theorem]{Conjecture}
\newtheorem*{conjecture*}{Conjecture}

\theoremstyle{examplestyle}
\newtheorem{example}[theorem]{Example}
\newtheorem{definition}[theorem]{Definition}
\newtheorem{definition-lemma}[theorem]{Definition-Lemma}
\newtheorem{definition*}{Definition}

\newtheorem{remark}[theorem]{Remark}
\newtheorem{remark*}{Remark}
\newtheorem{convention}[theorem]{Convention}

\theoremstyle{algorithmstyle}
\newtheorem{algorithm}[theorem]{Algorithm}

\newcommand{\QQ}{\mathbb{Q}}
\newcommand{\ZZ}{\mathbb{Z}}
\newcommand{\FF}{\mathbb{F}}
\newcommand{\PP}{\mathbb{P}}
\newcommand{\JC}{\mathrm{J}(C)}

\newcommand*{\myovline}[2]{\overbracket[#2][-1pt]{#1}}

\DeclareMathOperator{\Span}{Span}
\DeclareMathOperator{\Sym}{Sym}

\DeclareMathOperator{\Pic}{Pic}
\DeclareMathOperator{\Res}{Res}
\DeclareMathOperator{\Aut}{Aut}
\DeclareMathOperator{\res}{res}
\begin{document}

\title{Tritangents and their space sextics}

\author{Turku Ozlum Celik}
\address{Max Planck Institute MIS Leipzig}
% \curraddr{}
\email{tuerkue.celik@mis.mpg.de}
% \thanks{}

\author{Avinash Kulkarni}
\address{Simon Fraser University}
% \curraddr{}
\email{avi\_kulkarni@sfu.ca}
% \thanks{}

\author{Yue Ren}
\address{Max Planck Institute MIS Leipzig}
% \curraddr{}
\email{yue.ren@mis.mpg.de}
% \thanks{}

\author{Mahsa Sayyary Namin}
\address{Max Planck Institute MIS Leipzig}
% \curraddr{}
\email{mahsa.sayyary@mis.mpg.de}
% \thanks{}

\subjclass[2010]{14Q05 (primary), 14H50 (secondary)}
\keywords{algebraic curves, del Pezzo surfaces, theta characteristics, space sextic, tritangents}

\date{\today}

\maketitle
\begin{abstract}
  Two classical results in algebraic geometry are that the branch curve of a del Pezzo surface of degree 1 can be embedded as a space sextic curve in~$\mathbb P^3$ and that every space sextic curve has exactly $120$ tritangents corresponding to its odd theta characteristics. In this paper we revisit both results from the computational perspective. Specifically, we give an algorithm to construct space sextic curves that arise from blowing up $\PP^2$ at eight points and provide algorithms to compute the $120$ tritangents and their Steiner system of any space sextic.
  Furthermore, we develop efficient inverses to the aforementioned methods. We present an algorithm to either reconstruct the original eight points in $\PP^2$ from a space sextic or certify that this is not possible. Moreover, we extend a construction of Lehavi \cite{Lehavi15} which recovers a space sextic from its tritangents and Steiner system. All algorithms in this paper have been implemented in \textsc{magma}.
\end{abstract}

\section{Introduction}

Space sextic curves offer a rich example for understanding the various geometric features of non-planar algebraic curves. A space sextic curve $C$, abbreviated to space sextic, is a smooth algebraic curve in $\PP^3$ which is the intersection of a quadric and a cubic surface. Any space sextic is a non-hyperelliptic genus $4$ curve and conversely, the canonical model of any non-hyperelliptic genus $4$ curve is a space sextic. Several questions for plane curves have natural analogues for curves in arbitrary projective spaces. An important direction of inquiry is the study of hyperplanes that have a special intersection with the curve. In this article we focus on the tritangents of~$C$, planes in $\PP^3$ which are tangent to $C$ at every point of intersection with $C$. They reflect important intrinsic facts about the curve as well as important details about the extrinsic geometry.

For instance, the tritangents arise from effective representatives of the theta characteristics. As such, they provide insight into the Jacobian variety and over the complex numbers yield important information regarding values of theta functions. Conversely, any principally polarized abelian variety $A$ of dimension $4$ defines $120$ planes in $\PP^3$ via its theta functions. If $A$ is the Jacobian of a non-hyperelliptic genus $4$ curve, these $120$ planes are the tritangents of its canonical model \cite{CKS17}. On the extrinsic side of geometry, Caporaso and Sernesi~\cite{CS03} have proven that over the complex numbers the theta characteristics determine the curve uniquely, and Lehavi~\cite{Lehavi15} explained how over the complex numbers the tritangents can be used to recover the curve for generic space sextics on smooth quadrics. Additionally, Ranestad and Sturmfels \cite{RS12} have shown that the tritangents play an important role in the convex hull of compact space curves, whereas Kummer \cite{Kummer18} has used them to prove lower bounds on the number of faces of the convex hull.

In this article we present several algorithms related to space sextic curves and their tritangents.
In Section \ref{sec:genericSpaceSextic} we show how to compute the tritangents of space sextics on smooth quadrics and their Steiner systems; see Algorithms~\ref{alg:tritangentsGeneric} and~\ref{alg:steinerSystem}. In Section \ref{sec: space sextics from dp surface} we do the same for curves which arise from a construction involving del Pezzo surfaces of degree $1$. In Section \ref{sec:IdSSfromDP}, we show that these curves are exactly the ones lying on singular quadrics and answer \cite[Question 5]{KRSS17} by presenting the inverse of that construction, see Algorithm~\ref{algo: reconstruct eight points}.

Finally, in Section \ref{sec: Lehavi} we state a minor correction of \cite[Theorem 2]{Lehavi15} and extend \cite[Theorem 1 and 2]{Lehavi15} to space sextics on singular quadrics and over more general fields. Using them, we explain how to reconstruct space sextics from their tritangents, see Algorithms~\ref{algo: Lehavi1} and~\ref{algo: Lehavi2}, for generic curves on smooth quadrics and generic curves on singular quadrics. Moreover, we allow fields with positive characteristics, provided that it is sufficiently high. All algorithms have been implemented in \textsc{magma} \cite{magma} and are available on \url{https://software.mis.mpg.de}.

One of the greatest challenges in working with tritangents of space sextics symbolically, especially over fields of characteristic $0$, is that the field over which all of their equations are defined is monstrous. To be precise, if $C$ is a generic curve defined over $\QQ$, then its $120$ tritangents are defined over an algebraic extension of $\QQ$ of degree $\#\operatorname{Sp}(8, \FF_2) = 47377612800$. This is critical for the reconstruction algorithms, as their correctness hinge on the existence of a single successful example.
This is where space sextics from del Pezzo surfaces of degree $1$ come in. For them we have a reliable method of constructing space sextics with $120$ rational tritangents.

%%*********************************************************************************
%% SECTION
%%*********************************************************************************

\section{Definitions and notations}

We briefly run through some basic notions that will be of immediate relevance to us. In particular, we recall some basic facts on space sextics and their tritangents as well as Steiner complexes thereon. Our notation will closely follow that of \cite{Lehavi15}.

\begin{convention}
	For the remainder of the article, fix a base field $k$ of characteristic not equal to $2$. In Sections~\ref{sec:genericSpaceSextic} and~\ref{sec: Lehavi} we assume for convenience that $k$ is algebraically closed.

	Moreover, let $\mathbb P^3$ denote the projective $3$-space over $k$ and let $C\subseteq \PP^3$ be a smooth curve of degree $6$, or equivalently a canonical model of a non-hyperelliptic curve of genus $4$. We can write the curve as an intersection of a unique quadric surface $Q_C$ and a suitable cubic surface, neither of which are necessarily smooth.
\end{convention}

\subsection{General notions}
We denote the canonical divisor class of the smooth curve $C$ by $\kappa_C$. A \emph{theta characteristic} of $C$ is a divisor class $\theta$ such that $2\theta = \kappa_C$. A theta characteristic is \emph{odd} if $\dim H^0(C,\theta)$ is odd and is \emph{even} otherwise. An even theta characteristic is said to be \emph{vanishing} if $\dim H^0(C,\theta) > 0$.

For a smooth proper surface $X$, we use the terminology of \emph{intersection pairing}, \emph{exceptional curve}, and \emph{blow-up/blow-down} from \cite[Chapter V]{Hartshorne77}. A \emph{del Pezzo surface X of degree~1} is a smooth sextic surface in $\PP(1\!:\!1\!:\!2\!:\!3)$ with weighted coordinates $(s\!:\!t\!:\!w\!:\!r)$ of the form
\[
X \colon \lambda r^2 = f_0 w^3 + f_2(s,t) w^2 + f_4(s,t)w + f_6(s,t),
\]
where each $f_d$ is homogeneous polynomial of degree $d$ and $\lambda, f_0 \in k$ are nonzero. Without loss of generality, one may always assume that $\lambda=f_0=1$.

A \emph{Cayley cubic} is an irreducible cubic surface in $\PP^3$ whose singular locus consists of $4$ simple nodes in general position. If $k$ is algebraically closed there is exactly one Cayley cubic up to linear transformations of $\PP^3$.

\subsection{Space sextics, tritangents and Steiner systems}\label{sec: background}

Since $C$ is a smooth space sextic, it has exactly $120$ odd theta characteristics and there is no theta characteristic $\theta$ such that $\dim H^0(C,\theta) > 2$ (as $\dim H^0(C,2\theta) = 4$). In particular, each odd theta characteristic $\theta$ has a unique effective representative which we will refer to as $D_\theta$. Since $C$ is its own canonical model, by definition we have that $2D_\theta$ is cut out by linear form which we will denote $\l_\theta$.

\begin{definition}\label{def: tritangent}
	Let $\theta$ be an odd theta characteristic of $C$. We call the zero locus $H_\theta:=\mathcal Z(l_\theta) \subseteq \PP^3$ the \emph{tritangent plane} associated to $\theta$. Any plane $H\subseteq\PP^3$ of this form is called a \emph{tritangent plane} of $C$ or \emph{tritangent} for short.
\end{definition}

Note that we exclude from consideration planes which are tangent to $C$ at $3$ points (counted with multiplicity) but do not arise from odd theta characteristics. If the unique quadric $Q_C$ containing $C$ is smooth then every plane tangent to $C$ at $3$ points is a tritangent plane in the sense of Definition \ref{def: tritangent}. However, we deliberately consider curves in Section \ref{sec: space sextics from dp surface} where this is not the case.

\begin{definition}\label{def:SteinerSystem}
	Let $\JC$ be the Jacobian variety of $C$ and let $\JC[2]$ be its $2$-torsion points. For any $\alpha\in \JC[2]\setminus\{0\}$, we define the \emph{Steiner complex} associated to $\alpha$ by
	\[
	\Sigma_{C,\alpha} := \Big\{ \{\theta,\theta+\alpha\} : 2\theta = \kappa_C \text{ and } \dim H^0(C,\theta) = \dim H^0(C,\theta+\alpha)\equiv 1 \text{ mod } 2 \Big\}
	\]
	or equivalently, for the corresponding linear forms
	\[
	\mathcal{S}_{C,\alpha} := \big\{ \{l_{\theta},l_{\theta+\alpha}\} : \{ \theta, \theta + \alpha \} \in \Sigma_{C,\alpha} \big\}.
	\]
	We call the set $\mathcal{S}_C:=\{\mathcal{S}_{C,\alpha} : \alpha\in \JC[2]\setminus \{0\}\}$ the \emph{Steiner system} associated to $C$.
\end{definition}

\begin{definition}
	We call four theta characteristics $\theta_1,\ldots,\theta_4$ \emph{syzygetic} if $\sum_{i=1}^4 D_{\theta_i}$ is cut out by a quadric in $\PP^3$. This means that there exists a quadric $\mathcal{Z}(q) \subseteq\PP^3$ such that $C\cap (\bigcup_{i=1}^4 H_i)=C\cap \mathcal{Z}({q^2})$. We call the linear forms $l_1,\ldots,l_4$ \emph{syzygetic} if they arise from syzygetic odd theta characteristics.
\end{definition}

Note that if $\theta_1, \ldots, \theta_4$ are odd theta characteristics of $C$ such that $Y := C \cap (\bigcup_{i=1}^4 H_i)$ consists of $12$ points of multiplicity $1$, then $\theta_1, \ldots, \theta_4$ are syzygetic if and only if there exists a quadric $Q'$ which is not $Q_C$ passing through the $12$ points of $Y$.

\begin{remark}~\label{rem:SteinerSystem}
	There are exactly $255$ non-trivial $2$-torsion points $\alpha\in \JC[2]\setminus\{0\}$. On the other hand, the $120$ tritangents yield precisely $\binom{120}{2}=7140$ pairs $\{l_{\theta_1},l_{\theta_2}\}$.
	Intruigingly, the $7140$ pairs can be split evenly into $255$ blocks, each block containing exactly $28$ pairs, such that
	\begin{itemize}%[leftmargin=*]
		\item the union of any two pairs inside a block is syzygetic,
		\item the union of any two pairs in two distinct blocks is not syzygetic,
		\item the intersection of any two pairs inside a block is empty.
	\end{itemize}
	These blocks are the Steiner complexes $\Sigma_{C,\alpha}$ in Definition~\ref{def:SteinerSystem}. For further details see \cite[Section 5.4.2]{Dol2012}.

	In general combinatorics a Steiner system with parameters $(t,k,n)$ is an $n$-element set together with a set of $k$-element subsets called blocks, such that each $t$-element subset is contained in exactly one block. If we forget the pairings in our Steiner complexes $\mathcal S_{C,\alpha}$ and simply consider them as subsets of cardinality $56$, then our Steiner system $\mathcal S_C$ would be a specific instance of a general Steiner system with parameters $(t,k,n) = (2, 56, 120)$.
\end{remark}

%%*********************************************************************************
%% SECTION
%%*********************************************************************************

\section{Space sextics on smooth quadrics}\label{sec:genericSpaceSextic}
In this section we give a full description of the algorithm appearing in \cite[Section 4]{KRSS17} for calculating the $120$ tritangents of a space sextic that lies on a smooth quadric. Moreover, we provide an algorithm for calculating its Steiner system.

If $C\subseteq\PP^3$ lies on a smooth quadric, it suffices to compute all planes which are tangent at their points of intersection since there are no vanishing theta characteristics.
In theory, this can be done by intersecting $C$ with a parametrized plane and forcing its projection onto a single generic line to be three double points. In practice, it is often sufficient and generally much more efficient to project onto all coordinate axes instead.

\vfill\pagebreak
\begin{algorithm}[Tritangents of space sextics on smooth quadrics]\label{alg:tritangentsGeneric}\
	\begin{algorithmic}[1]
		\REQUIRE{$(f,g)$, where
			\begin{itemize}[leftmargin=*]
				\item $f,g\in k[\mathbf{x}]:=k[x_0,\ldots,x_3]$ homogeneous of degrees $2,3$,
				\item $\mathcal{Z}(f)$ smooth,
				\item $C:=\mathcal{Z}(f)\cap \mathcal{Z}(g)$ a space sextic.
		\end{itemize}}
		\ENSURE{$\{l_\theta : \theta \text{ odd theta characteristic of } C\}\subseteq k[\mathbf{x}]$.}
		\STATE Denote by $\PP k[t_0,t_1]_6$ the space of homogeneous sextic polynomials in $k[t_0,t_1]$. Let $\mathcal{Z}(J) \subseteq \PP k[t_0,t_1]_6$ be the threefold of perfect squares whose ideal $J\subseteq k[a_0,\ldots,a_6]$ is minimally generated by $45$ quartics; see \cite[Table 1]{LS16}
		\begin{center}
			\begin{tikzpicture}
			\node[anchor=base west] (Kt) {$\PP k[t_0,t_1]_6$};
			\node[anchor=base west,xshift=6cm] (P6) at (Kt.base east) {$\PP^6$};
			\draw[->] (Kt) -- node[above,xshift=2.5mm,font=\scriptsize] {$\cong$} node[below,font=\footnotesize] {$a_0t_0^6+\ldots+a_6t_1^6\longmapsto (a_0\!:\!\ldots\!:\!a_6)$} (P6);
			\node[anchor=base,yshift=-15mm] (squaresLeft) at (Kt.base) {$\{(b_0t_0^3+\ldots+b_3t_1^3)^2: (b_0\!:\!...\!:\!b_3)\in\PP^3\}$};
			\node[anchor=base,yshift=-15mm] (squaresRight) at (P6.base) {$\mathcal{Z}(J)$};
			\draw[->] (squaresLeft) -- node[above,font=\footnotesize] {$\cong$} (squaresRight);
			\draw[draw opacity=0] (squaresLeft.base) -- node[sloped] {$\subseteq$} ++(0,1.75);
			\draw[draw opacity=0] (squaresRight.base) -- node[sloped] {$\subseteq$} ++(0,1.75);
			\end{tikzpicture}
		\end{center}
		\STATE Consider the defining linear form of a parametrized plane in $k[\mathbf{x}]$
		\[ u_0x_0+\ldots+u_3x_3 \in k[\mathbf{u}^{\pm 1}][\mathbf{x}]:=k[u_0^{\pm 1},\ldots,u_3^{\pm 1}][x_0,\ldots,x_3].\vspace{-5mm} \]
		\FOR{$i,j\in\{0,\ldots,3\}$, $i\neq j$}
		\STATE Let $f_i,g_i$ be the images of $f, g$ under the substitution map
		\[k[\mathbf{u}^{\pm 1}][\mathbf{x}] \longrightarrow k[\mathbf{u}^{\pm 1}][x_j,x_k,x_l], \quad x_i \longmapsto \frac{-u_jx_j-u_kx_k-u_lx_l}{u_i}, \]
		where $\{i,j,k,l\}=\{0,\ldots,3\}$.
		\STATE Compute the following resultant which is a homogeneous sextic in $x_k,x_l$
		\[ \textstyle\sum_{\mu=0}^6 c_{\mu} \cdot x_k^{6-\mu} x_l^\mu := \Res(f_i,g_i,x_j)\in k[\mathbf{u}^{\pm 1}][x_k,x_l].\vspace{-5mm} \]
		\STATE Let $I_{ij}$ denote the image of $J$ under the substitution map
		\[ k[a_0,\ldots,a_6]\longrightarrow k[\mathbf{u}^{\pm 1}],\quad a_\mu \longmapsto c_{\mu}. \]
		\ENDFOR
		\STATE \label{step:idealI}Set $I:=\sum_{i,j} I_{ij}$ and compute the $120$ points in $\mathcal{Z}(I)$.
		\RETURN{$\{z_0x_0+\ldots+z_3x_3: (z_0\!:\!\ldots\!:\!z_3)\in \mathcal{Z}(I)\}$.}
	\end{algorithmic}
\end{algorithm}

A simple way to obtain the Steiner system from the set of linear forms is to exploit the fact that they form a Steiner system in the combinatorial sense; see Remark~\ref{rem:SteinerSystem}. Algorithm~\ref{alg:steinerSystem} simply enumerates through all pairs of linear forms, grouping those together which are syzygetic.

%\vfill\pagebreak
\begin{algorithm}[Steiner system via syzygetic relations]\label{alg:steinerSystem}\
	\begin{algorithmic}[1]
		\REQUIRE{$(f,g,{T})$, where
			\begin{itemize}[leftmargin=*]
				\item $f,g\in k[\mathbf{x}]$ homogeneous of degrees $2,3$,
				\item $\mathcal{Z}(f)$ smooth,
				\item $C:=\mathcal{Z}(f)\cap \mathcal{Z}(g)$ a space sextic,
				\item $T=\{l_\theta : \theta \text{ odd theta characteristic of } C\}\subseteq k[\mathbf{x}]$.
		\end{itemize}}
		\ENSURE{$\mathcal{S}_C$, the Steiner system associated to $C$.}
		\STATE Initialize $\mathcal S_C:=\emptyset$.
		\FOR{$\{l_1,l_2\}\subseteq T$ with $l_1\neq l_2$}
		\IF{$\exists S \in\mathcal S_C\;\; \exists \{l_3,l_4\}\in\Sigma: l_1,\ldots,l_4 \text{ syzygetic}$} \label{step:syzygetic}
		\STATE Set $S:= S\cup\{ \{l_1,l_2\} \}$.
		\ELSE
		\STATE Set $S:=\{ \{l_1,l_2\} \}$ and $\mathcal S_C:=\mathcal S_C\cup\{S\}$.
		\ENDIF
		\ENDFOR
		\RETURN{$\mathcal S_C$}.
	\end{algorithmic}
\end{algorithm}

\begin{remark}
	The bottleneck in Algorithm~\ref{alg:steinerSystem} is deciding whether four linear forms $l_1,\ldots,l_4$ are syzygetic in Step~\ref{step:syzygetic}. This can be done in several ways depending on what is viable. One straightforward option would be to compute whether the intersection $C\cap(\bigcup_{i=1}^4 \mathcal Z(l_i))$, as a divisor of $C$, is linearly equivalent to $2\kappa_C$ using the \textsc{magma} intrinsic \texttt{IsLinearlyEquivalent}.
\end{remark}

\begin{remark}
	Note that Algorithm~\ref{alg:steinerSystem} can also be applied ad verbum to space sextics on singular quadrics. Algorithm~\ref{alg:tritangentsGeneric} has a straightforward generalization:

	For a sextic on a singular quadric, the ideal $I$ in Step \ref{step:idealI} will be positive dimensional. It will contain infinitely many tangent planes of the singular quadric that contain the singularity; see Proposition~\ref{prop:oddThetaCharacteristics}. This positive dimensional component can be removed using saturation, leaving us with the desired $120$ tritangents that are of interest to us.
\end{remark}

When applied to space sextic curves that arise from the construction in Section \ref{sec: space sextics from dp surface}, both Algorithm~\ref{alg:tritangentsGeneric} and Algorithm~\ref{alg:steinerSystem} prove to be inferior to their more specialized counterparts in Section~\ref{sec: space sextics from dp surface}, which is why we generally advise against using them for space sextics on singular quadrics.

%%*********************************************************************************
%% SECTION
%%*********************************************************************************

\section{Space sextics on singular quadrics}\label{sec: space sextics from dp surface}

In this section we review the classical construction of space sextics from del Pezzo surfaces of degree $1$. These curves are special as the unique quadric containing it is singular. Nevertheless, the construction remains a valuable source of space sextics as many interesting qualities of them are both computationally and conceptually much more accessible compared to the generic case. In Proposition~\ref{prop:singularQuadricEvenTheta} we show that any space sextic on a singular quadric arises this way.

Moreover, we describe special algorithms which compute the $120$ tritangents as in \cite[Section 2]{KRSS17} and the Steiner system for this class of space sextics.

\subsection{Space sextics from del Pezzo surfaces of degree 1}\label{sec:ConsTheCurv}
Consider a configuration $\mathcal{P}:=\{P_1, \ldots ,P_8\}\subseteq\mathbb P^2$ of eight points in general position. First, note that the space of plane cubics through $\mathcal{P}$ has dimension $10-8=2$. Let $\{{s,t}\}$ be any basis of it. Next, observe that the space of plane sextics vanishing doubly on $\mathcal{P}$ has dimension $28-8 \cdot 3=4$. It is spanned by $\{s^2,st,t^2,{w}\}$ for some sextic $w$. Further, the space of plane nonics vanishing triply on $\mathcal{P}$ has dimension $55-8 \cdot 6=7$. It is spanned by $\{s^3,s^2t,st^2,t^3,sw,tw,r\}$ for some nonic $r$. This defines a rational map
\[
\begin{array}{rclc}
\psi\colon & \PP^2  &\dashrightarrow & \PP(1\!:\!1\!:\!2\!:\!3) \\
& (x\!:\!y\!:\!z)&\longmapsto & (s(x,y,z)\!:\!t(x,y,z)\!:\!w(x,y,z)\!:\!r(x,y,z)).
\end{array}
\]

\noindent
The closure of the image of $\psi$ is cut out by a single equation of degree six. As it turns out, we can say a lot about the image of $\psi$ \cite[Remark ~24.4.2]{manin1986cubic}.

\begin{proposition}
	The image of $\psi$ is a del Pezzo surface of degree 1.
\end{proposition}

Henceforth, we let $X := \myovline{\mathrm{Im}({\psi})}{0.6pt}$. By abuse notation we label the coordinates of $\PP(1\!:\!1\!:\!2\!:\!3)$ by $s,t,w,r$ (respectively, with respect to weights).
After a linear change of coordinates we can write the defining equation of $X$ in $\PP(1\!:\!1\!:\!2\!:\!3)$ as
\begin{equation*}\label{DelPezzoEq}
	X\colon r^2 = w^3+f_2(s,t) \cdot w^2+f_4(s,t) \cdot w + f_6(s,t),
\end{equation*}
such that each $f_2,f_4,f_6$ is homogeneous degree $2,4,6$ respectively. Note that we obtain an elliptic curve on $X$ by specializing $t=0$, which will be important in Section~\ref{sec: reconstruction algorithm}.

Now consider the projection $\pi\colon \PP(1\!:\!1\!:\!2\!:\!3) \rightarrow \PP(1\!:\!1\!:\!2)$ onto its first three coordinates and an embedding $\phi$ of $\PP(1\!:\!1\!:\!2)$ as a singular quadric in $\PP^3$, which is illustrated in Figure~\ref{fig:diagram1}.
\begin{figure}[h]
	\centering
	\begin{tikzpicture}
	\node (P2) {$\PP^2$};
	\node[anchor=base west,xshift=1cm] (P1123) at (P2.base east) {$\PP(1\!:\!1\!:\!2\!:\!3)$};
	\node[anchor=base west,yshift=-25mm] (P112) at (P1123.base west) {$\PP(1\!:\!1\!:\!2)$};
	\node[anchor=base west,yshift=-25mm] (P3) at (P112.base west) {$\PP^3$};
	\draw[->,dashed] (P2) -- node[above] {$\psi$} (P1123);
	\draw[->] ($(P1123.south west)+(0.25,0)$) -- node[left] {$\pi$} node[right]
	{%
		\begin{tikzpicture}[every node/.style={font=\scriptsize}]
		\node (P1123elem) {$(s\!:\!t\!:\!w\!:\!r)$};
		\draw[|->] ($(P1123elem.south west)+(0.75,0)$) -- ++(0,-0.5);
		\node[anchor=base west,yshift=-10mm] at (P1123elem.base west) {$(s\!:\!t\!:\!w)$};
		\end{tikzpicture}%
	}
	++(0,-1.7);
	\draw[->] ($(P112.south west)+(0.25,0)$) -- node[left] {$\phi$} node[right]
	{%
		\begin{tikzpicture}[every node/.style={font=\scriptsize}]
		\node (P112elem) {$(s\!:\!t\!:\!w)$};
		\draw[|->] ($(P112elem.south west)+(0.75,0)$) -- ++(0,-0.5);
		\node[anchor=base west,yshift=-10mm] at (P112elem.base west) {$(s^2\!:\!st\!:\!t^2\!:\!w)$};
		\end{tikzpicture}%
	}
	++(0,-1.7);
	\node[anchor=base west,xshift=10mm] (X) at (P1123.base east) {$X:=\psi(\mathbb P^2)$};
	\draw[draw opacity=0] (X) -- node[sloped] (supsetTop) {$\supseteq$} (P1123);
	\node[anchor=base west,yshift=-25mm] (Cprime) at (X.base west) {$C':=\text{BranchCurve}(\pi|_X)$};
	\node[anchor=base,yshift=-25mm] (supsetMid) at (supsetTop.base) {$\supseteq$};
	\node[anchor=base west,yshift=-25mm] (C) at (Cprime.base west) {$C$};
	\node[anchor=base,yshift=-25mm] (supsetLow) at (supsetMid.base) {$\supseteq$};
	\draw[->] ($(Cprime.south west)+(0.25,-0.1)$) -- ++(0,-1.6);
\end{tikzpicture}\vspace{-5mm}
\caption{The del Pezzo surface $X$ of degree $1$ and the branch curve $C'$.}
\label{fig:diagram1}
\end{figure}

The projection $\pi$ is a generically 2-to-1 rational map branched along the curve $C'$ of weighted degree six given by $r=0$. The defining equation of $C'$ in $\PP(1\!:\!1\!:\!2)$ is hence given by
\begin{equation}\label{eq:branch curve}
C'\colon \quad 0 = w^3+f_2(s,t) \cdot w^2+f_4(s,t) \cdot w+f_6(s,t). \tag{$\ast$}
\end{equation}
The embedding $\phi$ is an isomorphism between $\PP(1\!:\!1\!:\!2)$ and $\mathcal{Z}(x_1^2-x_0x_2) \subseteq \PP^3$, where $(x_0\!:\!x_1\!:\!x_2\!:\!x_3)$ are the homogeneous coordinates of $\PP^3$. The curve $C:=\phi(C')$ is therefore a space sextic which lies on a singular quadric.

\begin{example}\label{ex:delPezzoCurve}
Consider the configuration $\mathcal P=\{P_1,\ldots,P_8\}\subseteq\PP\ZZ^3$ of the following $8$ points as in \cite[Example~2.2]{KRSS17}:
\begin{align*}
	P_1& = (1\!:\!0\!:\!0), &P_5&=(10\!:\!11\!:\!1),&
	P_2& = (0\!:\!1\!:\!0), &P_6&=(19\!:\!-11\!:\!12),\\
	P_3& = (0\!:\!0\!:\!1), &P_7&=(15\!:\!19\!:\!-20),&
	P_4& = (1\!:\!1\!:\!1), &P_8&=(27\!:\!2\!:\!17).
\end{align*}
One can verify that the points in $\mathcal P\subseteq \PP \QQ^3$ are in general position, which also implies that the points in $\overline{\mathcal P}\subseteq \PP\FF_p^3$ remain in general position for all but finitely many primes~$p$. For instance, for $p=97$ they stay in general position, while for $p=5$ we have $\overline P_2 = \overline P_5 \in \PP \FF_5^3$.

From now on, consider the base field $k:=\mathbb F_{97}$. We compute two linearly independent cubics $s,t$ that vanish on $\mathcal P$ and a sextic $w$ not divisible by $s$ or $t$, vanishing doubly on $\mathcal P$. Our implementation gives
\begin{align*}
	s&=x_0x_1^2+33x_0^2x_2+74x_0x_1x_2+77x_1^2x_2+47x_0x_2^2+59x_1x_2^2,\\
	t&=x_0^2x_1+86x_0^2x_2+57x_0x_1x_2+35x_1^2x_2+15x_1x_2^2,\\
	w&=x_0^4x_1x_2+72x_0^3x_1^2x_2+68x_0^2x_1^3x_2+55x_0x_1^4x_2+67x_0^4x_2^2\\
	&\quad+78x_0^3x_1x_2^2+11x_0^2x_1^2x_2^2+18x_0x_1^3x_2^2+43x_1^4x_2^2+23x_0^2x_1x_2^3\\
	&\quad+55x_0x_1^2x_2^3+64x_1^3x_2^3+7x_0^2x_2^4+96x_0x_1x_2^4+21x_1^2x_2^4.
\end{align*}
This results in the following space sextic $C\subseteq\PP^3$, which can be verified to be smooth:
\begin{equation*}
	C\colon
	\begin{cases}
		0 = x_1^2-x_0x_2, \\
		0 = 94x_0^3+88x_0^2x_1+27x_0^2x_2+16x_0x_1x_2+82x_0x_2^2+72x_1x_2^2+73x_2^3\\
		\qquad+18x_0^2x_3+17x_0x_1x_3+84x_0x_2x_3+43x_1x_2x_3+37x_2^2x_3+63x_0x_3^2\\
		\qquad+64x_1x_3^2+63x_2x_3^2+74x_3^3.
	\end{cases}
\end{equation*}
\end{example}

\subsection{Constructing the odd tritangents}
Note that any tangent plane of the singular quadric $\mathcal{Z}(x_1^2-x_0x_2) \subseteq \PP^3$ is tangent to the curve $C$ at its points of intersection; the plane intersects the quadric in a double line passing through the node of the cone, which in turn intersects the cubic in three double points.
However, because all intersection points lie on a line through the node of the cone, the tangent plane arises from an even theta characteristic rather than an odd theta characteristic, which is why we will disregard them in the context of this paper.

\begin{proposition}\label{prop:oddThetaCharacteristics}
Let $C$ be a space sextic curve on a singular quadric $Q$. Then any plane corresponding to an odd theta characteristic of $C$ does not pass through the singularity of $Q$.
\end{proposition}
\begin{proof}
Let $H$ be a plane passing through the node of $Q$ such that $H$ is tangent to $C$ at each point in $C \cap H$. As $H$ passes through the node of $Q$, we have that $H \cap Q$ is the union of two lines which each meet $C$ in three points. As each point in $C \cap H$ has even multiplicity, the two lines must be coincident. But then the class of the divisor $\frac{1}{2}(C \cap H)$ of $C$ is the vanishing even theta characteristic.
\end{proof}

It turns out that the $120$ tritangents of $C$ have a natural description in the framework of our construction. Recall that $l_\theta$ is the linear form corresponding to the odd theta characteristic $\theta$. If $H_\theta := \mathcal{Z}(l_\theta)$ is a tritangent of $C$, then $\phi^{-1}(H_\theta)$ is a curve in $\PP(1\!:\!1\!:\!2)$ of weighted degree $2$ tangent to $C'$ at $3$ points (counting multiplicity).
\begin{definition}
Let $C'$ be a curve of the form in Equation (\ref{eq:branch curve}) and let $\phi$ be an embedding of $\PP(1\!:\!1\!:\!2)$ into $\PP^3$. We call $\phi^{-1}(H_\theta)$ a \emph{tritangent curve} of $C'$.
\end{definition}

Every tritangent curve of $C'$ is the image of exactly two exceptional curves $(e,e')$ on the del Pezzo surface $X$ under $\pi$, which are conjugate under the Bertini involution of $X$. We differentiate between the following types of tritangents, which are named after the degrees of $\overline{\psi^{-1}e},\overline{\psi^{-1}e'}\subseteq\PP^2$ \cite{KRSS17}:

\begin{enumerate}
	\item[(0,6)]
	$e$ and $e'$ are the exceptional fiber over a point $P_i$ and the strict transform of the sextic vanishing triply at $P_i$ and doubly at the other seven points.
	\item[(1,5)]
	$e$ and $e'$ are the strict transforms of the line through $\{P_i, P_j\}$ and the quintic vanishing at all eight points and doubly at the six points in $\mathcal{P} \backslash \{P_i, P_j\}$ (see Figure~\ref{fig:types} left).
	\item[(2,4)]
	$e$ and $e'$ are the strict transforms of the conic through $\mathcal P\setminus \{P_i,P_j,P_k\}$ and the quartic vanishing at $\mathcal{P}$ and doubly at $P_i,P_j,P_k$ (see Figure~\ref{fig:types} middle).
	\item[(3,3)]
	$e$ and $e'$ are the strict transforms of the cubic vanishing doubly at $P_i$, non-vanishing at  $P_j$, and vanishing singly at $\mathcal{P} \backslash \{P_i,P_j\}$ and the cubic vanishing doubly at $P_j$, non-vanishing at $P_i$, and vanishing singly at $\mathcal{P} \backslash \{P_i,P_j\}$ (see Figure~\ref{fig:types} right).
\end{enumerate}

\begin{figure}[h]
\centering
\begin{tikzpicture}
\node (middle) {
	\begin{tikzpicture}[scale=0.2,every node/.style={font=\footnotesize}]
	\draw[red] plot[smooth cycle, tension=1] coordinates { (0,0) (12,0) (12,3) (0,-9) (3,-9) (3,3) };
	\draw plot[smooth cycle, tension=1] coordinates { (-1,1) (13.5,1) (13.5,2.5) (0,2.5) };
	\fill[red] (3.45,-0.55) circle (8pt)
	node[anchor=north east,xshift=1mm] {$P_i$};
	\fill[red] (7.85,-0.55) circle (8pt)
	node[anchor=north west,xshift=-1mm] {$P_j$};
	\fill[red] (3.55,-4.95) circle (8pt)
	node[anchor=west] {$P_k$};
	\fill (-0.7,1) circle (8pt);
	\fill (0,2.45) circle (8pt);
	\fill (13.7,1) circle (8pt);
	\fill (14.45,2.4) circle (8pt);
	\fill (11.55,2.6) circle (8pt)
	node[anchor=south east,xshift=5mm,yshift=1mm] {$\mathcal P\setminus\{P_i.P_j,P_k\}$};
	\draw[fill=white] (3.1,2.6) circle (8pt);
	\draw[fill=white] (3.4,0.75) circle (8pt);
	\draw[fill=white] (9.35,0.75) circle (8pt);
	\end{tikzpicture}
};
\node[anchor=east] (left) at (middle.west) {
	\begin{tikzpicture}[scale=0.15,every node/.style={font=\footnotesize}]
	\draw[red] plot[smooth, tension=1] coordinates { (0,0) (4,-2) (9,0) (7,4) (4,2) (9,-2) (14,-1) (15,2) (11,1) (13,-5) (18,-7) (20,-4) (16,-3) (15,-9) (18,-12) (21,-10) (17,-8) (12,-12) (12,-16) (15,-17) (16,-14) (9,-13) (7,-15) (10,-17) (9,-11) (4,-11) (0,-18) };
	\draw (-2,-15) -- (22,-15);
	\fill[red] (7.5,-1.45) circle (8pt);
	\fill[red] (11.25,-2.35) circle (8pt);
	\fill[red] (14.75,-6.6) circle (8pt);
	\fill[red] (15,-8.75) circle (8pt)
	node[anchor=east,yshift=2mm] {$\mathcal P\setminus\{P_i,P_j\}$};
	\fill[red] (11.5,-12.9) circle (8pt);
	\fill[red] (10.05,-12.9) circle (8pt);
	\fill (16.75,-15) circle (8pt);
	\fill (7,-15) circle (8pt);
	\draw[fill=white] (1.15,-15) circle (8pt);
	\draw[fill=white] (10.35,-15) circle (8pt);
	\draw[fill=white] (11.45,-15) circle (8pt);
	\node[anchor=north east,xshift=1mm] at (7,-15) {$P_i$};
	\node[anchor=north west,xshift=-1mm] at (16.75,-15) {$P_j$};
	\end{tikzpicture}
};
\node[anchor=west,xshift=3mm] (right) at (middle.east) {
	\begin{tikzpicture}[scale=0.2,every node/.style={font=\footnotesize}]
	\draw[red] plot[smooth, tension=1] coordinates { (-2,-3) (6,12) (3,12) (9,0) (17,18) };
	\draw plot[smooth, tension=1] coordinates { (-3,-2) (12,6) (12,3) (0,9) (18,17) };
	\fill[red] (3.95,7.5) circle (8pt);
	\node[red,anchor=east] at (3.95,7.5) {$P_i$};
	\fill (7.5,3.95) circle (8pt);
	\node[anchor=north] at (7.5,3.95) {$P_j$};
	\fill[black!40] (-0.5,-0.5) circle (8pt)
	node[anchor=north west,xshift=-1mm] {$\mathcal P\setminus\{P_i,P_j\}$};
	\fill[black!40] (5.4,2.8) circle (8pt);
	\fill[black!40] (13.2,6.2) circle (8pt);
	\fill[black!40] (11.6,3.1) circle (8pt);
	\fill[black!40] (4.7,4.7) circle (8pt);
	\fill[black!40] (2.9,5.4) circle (8pt);
	\draw[fill=white] (3.1,11.6) circle (8pt);
	\draw[fill=white] (6.2,13.1) circle (8pt);
	\draw[fill=white] (16.55,16.55) circle (8pt);
	\end{tikzpicture}
};
\end{tikzpicture}\vspace{-7mm}
\caption{pairs of exceptional curves of type $(1,5), (2,4)$, and  $(3,3)$}
\label{fig:types}
\end{figure}
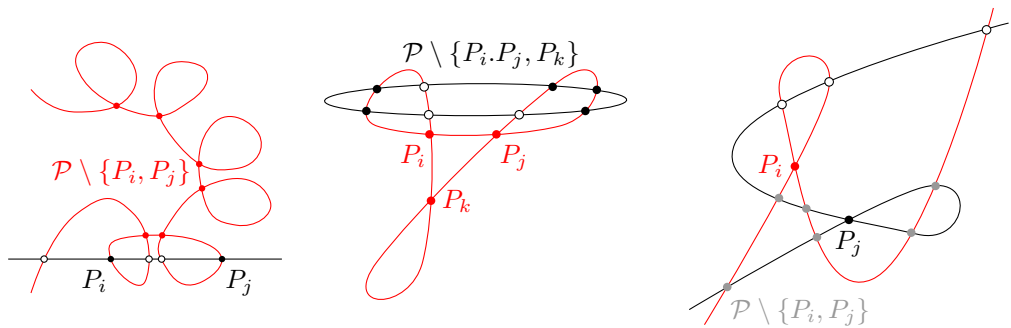

Of the $120$ tritangents, there are $8$, $\binom{8}{2}$, $\binom{8}{3}$ and $\binom{8}{2}$ tritangents of type $(0,6)$, $(1,5)$, $(2,4)$ and $(3,3)$ respectively. For tritangents of type $(0,6)$, the three components of the tangent cone of $e'$ determine the three contact points. For tritangents of the latter three types, the two preimages intersect in three points outside of $\mathcal P$. These are mapped to the contact points.

\begin{example}\label{ex:delPezzoTritangents}
Consider again the points $\mathcal P$ and the curve $C\subseteq \mathbb P \FF_{97}^3$ from Example~\ref{ex:delPezzoCurve}. Let $h_{12}$ and $h_{34}$ be the lines through the points $P_1,P_2$ and $P_3,P_4$ respectively and let $h_{125}$ and $h_{345}$ be the conics through the points $\mathcal P\setminus \{P_1,P_2,P_5\}$ and $\mathcal P\setminus\{P_3,P_4,P_5\}$ respectively. They are given by
\begin{align*}
h_{12} &= x_2, & h_{125} &= 19x_0^2 + 58x_0x_1 + 65x_1^2 + 21x_0x_2 + 31x_1x_2,\\
h_{34} &= x_0-x_1, & h_{345} &= 3x_0x_1 + 89x_0x_2 + 3x_1x_2 + 42x_2^2.
\end{align*}
Their images in $\mathcal P^3$ are cut out by the following polynomials and one can verify that each of them intersect $C$ in three double points:
\begin{align*}
l_{12} &= x_3, &l_{125} &= x_0 + 52x_1 + 23x_2 + 87x_3,\\
l_{34} &= x_0 + 70x_1 + 8x_2 + 43x_3, &l_{345} &= x_0 + 71x_1 + 65x_2 + 94x_3.
\end{align*}
\end{example}

\subsection{Constructing the Steiner system}

Let $P_1, \dots, P_8$ be eight points in $\PP^2$ in general position, let $X$ be the del Pezzo surface of degree $1$ obtained by blowing up $\PP^2$ at $\{P_1, \ldots, P_8\}$, and let $C$ the space sextic lying on $X$ as in the Section \ref{sec:ConsTheCurv}. The Steiner system of $C$ can be much more efficiently constructed than for a generic space sextic by making use of the initial eight points in $\PP^2$:

We label each tritangent by a subset of $\{1,\ldots,9\}$ depending on their types as follows:
\[ L(H) :=
\begin{cases}
\{1,\ldots,8\}\setminus \{i\} & \text{if }H \text{ is of type }(0,6),\\
\{1,\ldots,8,9\}\setminus \{i,j\} & \text{if }H \text{ is of type }(1,5),\\
\{i,j,k\} &\text{if }H \text{ is of type }(2,4),\\
\{i,j,9\} &\text{if }H \text{ is of type }(3,3),
\end{cases}
\]
where for each type $i,j,k$ are as in Figure \ref{fig:types}. Given four theta characteristics $\theta_1,\ldots,\theta_4$ and their tritangents $H_1,\ldots,H_4$, we then have
\[ \theta_1,\ldots, \theta_4 \text{ syzygetic}\quad \Longleftrightarrow \quad L(H_1) \mathbin{\Delta} \ldots \mathbin{\Delta} L(H_4)=\emptyset. \]

The fact that the Steiner system of $C$ can be constructed purely combinatorially is rooted in a classical connection between theta characteristics of $C$ and quadratic forms over $\JC[2]$ \cite[Section 5]{Dol2012}:

The $256$ points in $\JC[2]$ together with the Weyl-pairing
% $\langle \cdot,\cdot\rangle$
can be regarded as a symplectic vector space over $\FF_2$ of dimension $8$. The space of quadratic forms of $\JC[2]$, which we denote by $\mathcal{Q}(\JC[2])$, is a homogeneous space of $\JC[2]$. The disjoint union $\JC[2] \coprod \mathcal{Q}(\JC[2])$ has the natural structure of an $\FF_2$-vector space of dimension $9$. Each quadratic form corresponds to a theta characteristic of $C$.

Let $\res_C\colon \Pic(X) \rightarrow \Pic(C)$ be the natural restriction and set $v_i:=\res_C(e_i+\kappa_X)$, where $e_1,\ldots,e_8$ are the eight exceptional divisors of $X$ corresponding to the eight points $P_1,\ldots,P_8$ under the blow-down map. We may construct an Aronhold basis $B$ of $\JC[2] \coprod \mathcal{Q}(\JC[2])$ from the set $\{ v_1, \ldots, v_8, \res_C(-\kappa_X) \}\subseteq\Pic(C)$ via \cite[Theorem II A1.1.]{RauFar1974}. The labeling of a tritangent plane stated above can be understood in terms of the $B$-coordinates of the corresponding odd theta characteristic, viewed as an element of $\JC[2] \coprod \mathcal{Q}(\JC[2])$. For further details, see \cite[Section 1.1.4]{Tur2018}.

\begin{example}
Consider again the curve $C\subseteq\PP \FF_{97}^3$ in Example~\ref{ex:delPezzoCurve} and its four tritangents from Example~\ref{ex:delPezzoTritangents}. Their labels are:
\begin{align*}
&L(\mathcal Z(l_{12})) = \{ 3,4,5,6,7,8,9\}, & L(\mathcal Z(l_{125})) = \{1,2,5\}. \\
&L(\mathcal Z(l_{34})) = \{1,2,5,6,7,8,9\}, & L(\mathcal Z(l_{345})) = \{3,4,5\}.
\end{align*}
They have an empty symmetric difference and are therefore syzygetic. Indeed, one can verify that there exists a quadric which is not $\mathcal{Z}(x_1^2-x_0x_2)$ that vanishes on all $12$ contact points of the tritangents with $C$.
\end{example}

%%*********************************************************************************
%% SECTION
%%*********************************************************************************

\section{Identifying space sextics from del Pezzo surfaces of degree one} \label{sec:IdSSfromDP}

Given a space sextic curve $C$ in $\PP^3$ which comes from blowing up the plane at eight points, we will construct some collection of $8$ points in $\PP^2$ such that the construction applied to these $8$ points gives a space sextic isomorphic to $C$.
If a del Pezzo surface $X$ of degree $1$ has $8$ pairwise orthogonal exceptional curves defined over $k$, then the blow-down of $X$ along these $8$ curves is isomorphic to $\PP^2$ \cite[Theorem 24.4.iii]{manin1986cubic}. The $8$ exceptional curves mark $8$ points in the plane. Since this is true of any collection of $8$ pairwise orthogonal exceptional curves of $X$, we see that the set of $8$ points in the plane which construct the branch curve of $X$ is not unique, even up to linear transformations. The goal of our algorithm is to find one such set.

Over an algebraically closed field of characteristic not $2$, a space sextic curve $C$ arises from a del Pezzo surface of degree $1$ if and only if the unique quadric in $\PP^3$ containing $C$ is singular \cite[Theorem ~24.4.iii]{manin1986cubic}. However, when $k$ is not algebraically closed the situation is slightly more complicated. Nevertheless, our implementation is able to detect when a curve comes from eight $k$-rational points in the plane. Our method of detection and recovery uses the $120$ tritangents given as input. By assuming that the $120$ odd theta characteristics are defined over $k$, we can return to the simplicity of the algebraically closed case.

\begin{proposition}\label{prop:singularQuadricEvenTheta}
	Then a space sextic curve $C$ defined over $k$ arises from a del Pezzo surface of degree $1$ if and only if the unique quadric in $\PP^3$ containing $C$ is singular and all $120$ odd theta characteristics are defined over $k$.
\end{proposition}
\begin{proof}
	The forward direction is entirely classical. For the reverse, let $Q$ be the singular quadric containing $C$. As $C$ is smooth, we must have that $Q$ is a quadric cone. We claim that $Q$ is isomorphic over $k$ to the weighted projective space ${\PP(1\!:\!1\!:\!2)}$. Let $H$ be any tritangent plane of $C$ corresponding to an odd theta characteristic and let $D := \frac{1}{2} (H \cap C)$. By Proposition~\ref{prop:oddThetaCharacteristics}, $H$ does not pass through the singularity of $Q$, so $Q \cap H$ is a plane conic and $D$ is an odd degree divisor on $Q \cap H$. Thus, we have that $\PP^1 \cong Q \cap H$.

	Since $Q \cap H \subseteq H$ is a plane conic with a $k$-rational point, we may change coordinates so that $H = \mathcal{Z}(x_3)$ and such that $H \cap Q = \mathcal{Z}(x_3, x_1^2 - x_0 x_2)$. That is, $Q$ is isomorphic to the standard quadric cone. Since $C$ lies on a standard quadric cone, we have that $C$ lies on a degree $1$ del Pezzo surface $X$. Moreover, as all 120 tritangents of $C$ are defined over $k$, there is a quadratic twist of $X$ such that all $240$ exceptional curves are defined over $k$. The proposition now follows from \cite[Theorem 24.4]{manin1986cubic}.
\end{proof}

\begin{remark}
	Note that the construction of the space sextics from del Pezzo surfaces of degree $1$ as in \cite[Section 2]{KRSS17} considers slightly more general configurations of $8$ points in $\PP^2$. Specifically, it was only required that the set of eight points be $\operatorname{Gal}(\bar \QQ / \QQ)$-invariant in order to construct a space sextic defined over $\QQ$. For the sake of simplifying our implementation we do not consider this general setting.
\end{remark}

% Subsection
\subsection{The reconstruction algorithm} \label{sec: reconstruction algorithm}

The algorithm proceeds along the following general steps:
\begin{enumerate}[leftmargin=*,label={(\roman*)}]
	\item\label{step:i}
	Compute a degree $1$ del Pezzo surface $X$ as in Figure~\ref{fig:diagram1}.
	\item\label{step:ii}
	Determine the $240$ exceptional curves on $X$, possibly requiring a quadratic twisting of $X$.
	\item\label{step:iii}
	Identify $8$ pairwise orthogonal curves defined over $k$.
	\item\label{step:iv}
	Identify a genus $1$ curve $E \subseteq X$ defined over $k$, and intersect it with the $8$ exceptional curves to obtain $8$ points.
	\item\label{step:v}
	Construct a particular embedding of $E$ into $\PP^2$ with the properties expected of the blow-down of $E$ along $8$ pairwise orthogonal exceptional curves.
\end{enumerate}
We shall assume that our curve $C$ is given as the intersection of a cubic $\mathcal{Z}(f)$ and the quadric $\mathcal{Z}(x_1^2-x_0x_2)$ in $\PP^3$ with coordinates $(x_0\!:\!x_1\!:\!x_2\!:\!x_3)$. Note that up to linear transformation any singular quadric with a smooth $k$-rational point is the one we have specified. We assume that we are given the $120$ tritangent planes as linear forms $\{\ell_1 , \ldots, \ell_{120} \}$ with coefficients in $k$. Finally, we use the existence of the maps in Figure~\ref{fig:diagram1} where $C$ would be our given curve.
%Finally, as we have fixed a specific quadric cone, we fix $\phi(s\!:\!t\!:\!w) := (s^2\!:\!st\!:\!t^2\!:\!w)$.
As before, we will denote the coordinates of $\PP(1\!:\!1\!:\!2)$ and $\PP(1\!:\!1\!:\!2\!:\!3)$ by $(s\!:\!t\!:\!w)$ and $(s\!:\!t\!:\!w\!:\!r)$ respectively and let $\pi$ be the usual projection.

To accomplish Steps \ref{step:i} and \ref{step:ii}, we compute the pullback of the curve and the tritangent planes to $\PP(1\!:\!1\!:\!2)$ to set an equation for the del Pezzo surface $X$ of degree 1. More precisely, let $F$ be the pullback of $f$ under the map $\phi$ and $h_i$ be the pullback of $l_i$ under the map $\phi$ for $i=1,\dots ,120$. We may normalize $F$ and the $h_i$ so that the leading coefficient in $w$ is $1$. We choose $\lambda \in k^\times$ such that $\Res(h_1(s,t,w), \lambda F(s,t,w), w) \in k[s,t]$ is a square over $k$. We set $r^2 - \lambda F(s,t,w)$ to be the defining equation of a del Pezzo surface $X$ in $\PP(1\!:\!1\!:\!2\!:\!3)$. By definition, $\phi$ maps the branch curve of the projection $\pi|_X$ to $C$.
%It has $C$ as its branch curve by definition.
Thus, $X$ is exactly the same surface occurring in Section~\ref{sec:ConsTheCurv} up to isomorphism over $\bar k$. The choice of $\lambda$ merely selects the unique quadratic twist such that the pair of exceptional curves over the tritangent curve defined by $h_1$ is split over $k$.

\begin{lemma}
	With the notation of the preceding paragraph, the subvariety
	\[ \mathcal{Z}(r^2 - \lambda F(s,t,w), h_i(s,t,w)) \subseteq X \]
	is the union of two Bertini-conjugate exceptional curves $e_{2i-1}, \ e_{2i}$. The curves are individually defined over $k$ and are explicitly given by
	\[ e_{2i-1}, \ e_{2i} :=   \mathcal{Z}\!\left(h_i(s,t,w), r \pm \sqrt{\Res(h_i(s,t,w), \lambda F(s,t,w), w) } \right). \]
\end{lemma}

\begin{proof}
	From Section~\ref{sec: space sextics from dp surface} we have that the image under $\pi$ of any exceptional curve is a tritangent curve and every tritangent curve arises in this way. The first statement of the lemma follows immediately.

	If $C$ arose from blowing up eight $k$-rational points in the plane, then all of the exceptional curves on some quadratic twist of $X$ must be defined over $k$. Since we have already chosen a twist where one pair of exceptional curves is split over $k$, all of them must be so.

	All that remains to be shown is the correctness of the defining equations. We have defined $h_i$ to be in the defining ideal of the tritangent curve, so the scheme
	\[
	\mathcal{Z}\big(h_i(s,t,w), r^2 - \Res(h_i(s,t,w), \lambda F(s,t,w), w) \big)
	\]
	must lie over the tritangent curve defined by $h_i$. Next, we see that
	\[
	X \cap \mathcal{Z}(h_i(s,t,w)) = \mathcal{Z}(r^2 - \lambda F(s,t,w), h_i(s,t,w))
	\]
	is reducible, with the two components defined over $k$ corresponding to the two exceptional curves over $\mathcal{Z}\!\left(h_i(s,t,w)\right) \subseteq \PP(1\!:\!1\!:\!2)$. Note $h_i(s,t,w)$ is a degree~$1$ polynomial in $w$ and both $h_i$ and $F$ are monic due to the normalization, so the resultant is really just the evaluation of $\lambda F$ at $-h_i(s,t,0)$. In particular, the resultant is a square and
	\[
	X \cap \mathcal{Z}(h_i(s,t,w)) = \mathcal{Z}\!\left(r^2 - \lambda F(s,t,-h_i(s,t,0) ), h_i(s,t,w)\right).
	\]
	The polynomial $r^2 - \lambda F(s,t,-h_i(s,t,0))$ factors over $k$, giving the two components.
\end{proof}

For Step \ref{step:iii}, we compute the Gram matrix for the intersection pairing on the lattice generated by the exceptional curves using Gr\"obner bases. We use the fact that the Weyl group of $E_8$ acts via isometry on the lattice generated by the exceptional curves to optimize the process of searching for a collection of $8$ pairwise orthogonal curves.

Step \ref{step:iv} and \ref{step:v} of our algorithm identify the eight points using the geometry of genus~$1$ curves, allowing us to avoid an expensive blow-down computation. Step \ref{step:iv} turns out to be rather easy; the genus $1$ curve $E$ is obtained by specializing $t=0$. To indirectly compute the blow-down of this genus $1$ curve, we use the following lemma.

\begin{lemma}
	Let $E$ be a plane cubic curve and let $p_1, p_2, p_3 \in E(\bar k)$ be three collinear points. Then the embedding of $E$ into $\PP^2$ is the unique embedding with ${p_1, p_2, p_3}$ collinear, up to $\Aut(\PP^2)$.
\end{lemma}

\begin{proof}
	Let $D := p_1 + p_2 + p_3$ and let $\gamma \colon E \rightarrow \PP^2$ be an embedding. Since $D$ is a hyperplane section of this embedding, we have that $\gamma$ factors into $|D|\colon E \rightarrow \PP^2$. There is exactly one complete linear system containing $D$, so we are done.
\end{proof}

The blow-down of the genus $1$ curve $E$ will be a plane curve of arithmetic genus~$1$, which is to say, a plane cubic curve. It suffices for us to identify three points on $E$ which are collinear in the blow-down to identify the correct embedding into $\PP^2$ given by the blow-down. Crucially, \emph{the eight intersection points of $E$ with the set of eight pairwise orthogonal exceptional curves are the points where $E$ meets the eight exceptional points of the blow-up in $\PP^2$}. In order to identify three collinear points on $E$, we use a ``trivial'' lemma, which states that the effective divisors in the hyperplane class correspond to hyperplanes.

\begin{lemma}
	Let $X$ be the degree 1 del Pezzo surface obtained from blowing up $\{P_1, \ldots, P_8 \}$. Let $E \subseteq \PP^2$ be a plane cubic curve passing through the eight points, let $\res_E\colon \Pic(X) \rightarrow \Pic(E)$ be the natural restriction, and let $\ell \in \Pic(X)$ be the hyperplane class on $X$. Then the effective representatives of the divisor class $\res_E(\ell)$ are exactly the divisors of $E$ defined by lines in $\PP^2$.
\end{lemma}

\begin{proof}
	By definition, a representative of the class $\ell$ is the strict transform of a line in $\PP^2$ not passing through any of the eight points of the blow up. The result extends to all effective representatives of $\res_E(\ell)$ via linear equivalence.
\end{proof}

An appropriate curve $E$, divisor $D$, and class $\ell$ can be identified using only $X$ and the exceptional curves of $X$. The class $\ell \in \Pic(X)$ is uniquely identified by its intersection numbers \cite[Corollary 25.1.1]{manin1986cubic}; it is the unique divisor class such that $\ell^2 = 1$, $\ell\cdot e_i = 0$, and (redundantly) $\ell \cdot(-\kappa_X) = 3$. To find a representative, we use the following lemma.

\begin{lemma}
	Let $e_1', e_2'$ be a pair of orthogonal exceptional curves of $X$. There exists an exceptional curve $e$ such that	$\ell = [e + (e_1' + e_2')]$.
\end{lemma}

\begin{proof}
	Note that $X$ is the blow-up of $\PP^2$ at $\{P_1, \ldots, P_8\}$. The strict transform of the line through $\{P_1,P_2\}$ is an exceptional curve $e$. Thus the pullback of the line through $\{P_1, P_2 \}$ by the blow-down, as a divisor of $X$, is $e + (e_1' + e_2')$.
\end{proof}

We are now ready to present Algorithm \ref{algo: reconstruct eight points}.

\vfill\pagebreak
\begin{algorithm}[Reconstruction of 8 points from space sextics] \label{algo: reconstruct eight points}\
	\begin{algorithmic}[1]
		\REQUIRE {$(f,\{l_1,\ldots,l_{120}\})$, where
			\begin{itemize}[leftmargin=*]
				\item $f\in k[x_0,\ldots,x_3]$ homogeneous of degrees $3$ such that $C:=\mathcal{Z}(f)\cap \mathcal{Z}(x_1^2-x_0x_2)$ space sextic,
				\item $\{l_1,\ldots,l_{120}\}$ is a list of 120 tritangent plane to $C$.
			\end{itemize}
		}
		\ENSURE{$\mathcal P\subseteq \PP^2$, a set of $8$ points such that the del Pezzo surface of degree $1$ obtained by blowing up $\PP^2$ at $\mathcal{P}$ has branch curve isomorphic to $C$}.
		\STATE Set $F(s,t,w) := f(s^2,st,t^2,w)$ to be the pullback of $f$ under the map $\phi$.
		\STATE Set $h_i := \ell_i(s^2,st,t^2,w)$ to be the pullback of  linear form $\ell_i$ under the map $\phi$.
		\STATE Normalize $F$ and $h_1, \ldots, h_{120}$.
		\STATE Choose $\lambda \in k^\times$ such that $\Res(h_1(s,t,w), \lambda F(s,t,w), w) \in k[s,t]$ is a square.
		\STATE Set the defining equation of the del Pezzo surface $X$ to be $r^2 - \lambda F(s,t,w)$.
		\STATE  Set the pair of exceptional curves corresponding to the $i$-th tritangent,
		\[ e_{2i-1}, \ e_{2i} :=   \mathcal{Z}\!\left(h_i(s,t,w), r \pm \sqrt{\Res(h_i(s,t,w), \lambda F(s,t,w), w) } \right).\vspace{-5mm} \] % for $1 \le i \le 120 $.
		\STATE Compute the Gram matrix $ M := (\deg (e_i \cap e_j))_{1 \leq i\neq j \leq 240}$, where the diagonal entries are set to $-1$.
		\STATE \label{step:choice}Determine a collection $\{e_i'\}$ of eight pairwise orthogonal exceptional curves by computing an $8\times 8$ principal submatrix $B$ of $M$ such that $-B$ is the identity matrix.
		\STATE Construct the elliptic curve $E$ on $X$ by setting $t=0$ and set $p_i := E \cap e_i'$.
		\STATE Identify the unique exceptional curve $e$ such that
		\begin{equation*}
			\deg (e \cap e_j') =
			\begin{cases}
				1 & \text{if $j = 1,2$}, \\
				0 & \text{otherwise}.
			\end{cases}\vspace{-3mm}
		\end{equation*}%
		\STATE \label{step:RiemannRoch}Compute $H^0(E, \mathcal{O}_E(D))$, where $D := (e \cap E) + p_1 + p_2$ is a divisor on $E$.
		\RETURN
		The images of $p_1, \ldots, p_8$ under the map defined by $H^0(E, \mathcal{O}_E(D))$.
	\end{algorithmic}
\end{algorithm}

\begin{remark}
	One challenging step in Algorithm~\ref{algo: reconstruct eight points} that we deliberately avoided expanding on is the computation of $H^0(E, \mathcal{O}_E(D))$ in Step~\ref{step:RiemannRoch}, since most computer algebra systems have existing commands for it. For example, in our \textsc{magma} implementation we use the intrinsic command \texttt{DivisorMap}.
\end{remark}

\begin{example}
	We apply our algorithm to the space sextic curve obtained in Example~\ref{ex:delPezzoCurve}. We recall that the defining equations of the space sextic $C$ are:
	\begin{align*}
		q &= x_1^2-x_0x_2, \\
		f &= 94x_0^3\!+\!88x_0^2x_1\!+\!27x_0^2x_2\!+\!16x_0x_1x_2\!+\!82x_0x_2^2\!+\!72x_1x_2^2\!+\!73x_2^3\!+\!18x_0^2x_3\!+\!17x_0x_1x_3\\
		&\qquad +\!84x_0x_2x_3\!+\!43x_1x_2x_3\!+\!37x_2^2x_3\!+\!63x_0x_3^2\!+\!64x_1x_3^2\!+\!63x_2x_3^2\!+\!74x_3^3.
	\end{align*}
	The defining equation of the del Pezzo surface $X$ in $\PP(1\!:\!1\!:\!2\!:\!3)$ with coordinates $(s\!:\!t\!:\!w\!:\!r)$ constructed in Algorithm \ref{algo: reconstruct eight points} is
	\begin{align*}
		X\colon -r^2 =&\; 2s^6 + 6s^5t + 79s^4t^2 + 54s^3t^3 + 10s^2t^4 + 49st^5 + 16t^6 + 85s^4w + 21s^3tw \\
		& + 41s^2t^2w + 36st^3w + 40t^4w + 55s^2w^2 + 22stw^2 + 55t^2w^2 + 80w^3.
	\end{align*}
	Algorithm \ref{algo: reconstruct eight points} computes some set of eight pairwise orthogonal exceptional curves. Below are the defining equations in $\PP(1\!:\!1\!:\!2\!:\!3)$.
	\begin{align*}
		e_1\colon &
		\begin{cases}
			w = 61 s^{2} + 19 s t + 83 t^{2},\\
			r = 82 s^{3} + 44 s^{2} t + 23 s t^{2} + 75 t^{3},
		\end{cases}
		& e_2\colon &
		\begin{cases}
			w = 41 s^{2} + 9 s t + 81 t^{2},\\
			r = 23 s^{3} + 13 s^{2} t + 14 s t^{2} + 37 t^{3},
		\end{cases}\\
		e_3\colon &
		\begin{cases}
			w = 45 s^{2} + 36 s t + 90 t^{2},\\
			r = 54 s^{2} t + 18 s t^{2} + 13 t^{3},
		\end{cases}
		& e_4\colon &
		\begin{cases}
			w = 4 s^{2} + 8 s t + 93 t^{2},\\
			r = 43 s^{3} + 5 s^{2} t + 22 s t^{2} + 58 t^{3},
		\end{cases}\\
		e_5\colon &
		\begin{cases}
			w = 26 s^{2} + 57 s t + 77 t^{2},\\
			r = 68 s^{3} + 55 s^{2} t + 30 s t^{2} + 95 t^{3},
		\end{cases}
		& e_6\colon &
		\begin{cases}
			w = - s^{2} + 22 s t + 66 t^{2},\\
			r = 81 s^{3} + 41 s^{2} t + 20 s t^{2} + 5 t^{3},
		\end{cases}\\
		e_7\colon &
		\begin{cases}
			w = - s^{2} + 69 s t + 27 t^{2},\\
			r = 16 s^{3} + 41 s^{2} t + 16 s t^{2} + 13 t^{3},
		\end{cases}
		& e_8\colon &
		\begin{cases}
			w = s t + 67 t^{2},\\
			r = 32 s^{3} + 40 s^{2} t + 66 s t^{2} + 24 t^{3}.
		\end{cases}
	\end{align*}
	\noindent
	The genus 1 curve $E \subseteq X$ is defined by the common zero set of $t$ and the defining equation of $X$. Finally, the resulting eight points $\mathcal P'=\{P_1',\ldots,P_8'\}$ are
	\begin{align*}
		P_1'&=(35\!:\!48\!:\!1),& P_2'& = (41\!:\!1\!:\!0),&
		P_3'&=(41\!:\!91\!:\!1),& P_4'& = (61\!:\!1\!:\!0),\\
		P_5'&=(11\!:\!14\!:\!1),& P_6'& = (27\!:\!95\!:\!1),&
		P_7'&=(52\!:\!80\!:\!1),& P_8'& = (88\!:\!68\!:\!1).
	\end{align*}
	Notice that the eight points $\mathcal P'$ are not the same as the original collection of eight points $\mathcal P$ in Example~\ref{ex:delPezzoCurve}.
	The difference potentially arises due to choice made at Step~(\ref{step:choice}) and linear transformations of $\PP^2$. In this case the linear transformation given by the following matrix maps $P_i'$ to $P_i$ for $i=1,\ldots,8$:
	\[ \begin{pmatrix}
	34 & 61 & 39 \\
	27 &  3 & 75 \\
	34 & 61 & 53
	\end{pmatrix}.\]
\end{example}

%%*********************************************************************************
%% SECTION
%%*********************************************************************************

\section{Steiner systems to space sextics} \label{sec: Lehavi}

In this section, we extend Lehavi's methods \cite{Lehavi15} for reconstructing space sextics from their Steiner system to space sextics on singular quadrics and to space sextics over more general fields. We detail their implementation in \textsc{magma} and, as a consequence, obtain a simplified proof of Lehavi's results.

The method is naturally divided in two parts: the reconstruction of the unique quadric surface containing the curve and the reconstruction of a cubic surface which cuts out the curve on the quadric.

\subsection{Reconstructing the unique quadric}

For the reconstruction of the quadric, we briefly recall the notation in~\cite[Section 1]{Lehavi15}.

\begin{definition}
	Let $V_C:=H^0\!\left(\mathcal O_{|\kappa_C|^\ast}(2)\right)$ be the vector space of degree $2$ forms on the canonical $\PP^3$ containing $C$.

	For all $\alpha\in \JC[2]\setminus\{0\}$ and all $\{\theta,\theta+\alpha\}\in\Sigma_{C,\alpha}$, let $q_{\theta,\theta+\alpha}:= l_{\theta}\cdot l_{\theta+\alpha}\in V_C$ be the quadric form cutting out the two tritangent planes corresponding to $\theta$ and $\theta+\alpha$.

	For all $\alpha\in \JC[2]\setminus\{0\}$, let $V_{C,\alpha}\subseteq V_C$ be defined as in the following diagram:
	\begin{center}
		\begin{tikzpicture}
		\node[anchor=base west] (VC) at (0,0) {$V_C$};
		\node[anchor=base west,xshift=-1mm] (eqVC) at (VC.base east) {$=$};
		\node[anchor=base west,xshift=-1mm] (HK2) at (eqVC.base east) {$H^0\!\left(\mathcal O_{|\kappa_C|^\ast}(2)\right)$};
		\node[anchor=north,yshift=-1cm] (H2K) at (HK2.south) {$H^0\!\left(2\kappa_C\right)$};
		\draw[->>] (HK2) -- node[right] {$p$} (H2K);
		\node[anchor=base west,xshift=2cm] (HKalpha2) at (H2K.base east) {$H^0\!\left(\mathcal O_{|\kappa_C+\alpha|^\ast}(2)\right)$};
		\draw[->] (HKalpha2) -- node[above] {$i$} (H2K);
		\node[anchor=south east,yshift=1cm] (piHKalpha2) at (HKalpha2.north east) {$p^{-1}i H^0\!\left(\mathcal O_{|\kappa_C+\alpha|^\ast}(2)\right)$};
		\draw[draw opacity=0] (piHKalpha2) -- node {$\supseteq$} (HK2);
		\node[anchor=base west,xshift=-1mm] (eqVCalpha) at (piHKalpha2.base east) {$=:$};
		\node[anchor=base west,xshift=-1mm] (VCalpha) at (eqVCalpha.base east) {$V_{C,\alpha}$};
		\end{tikzpicture}
	\end{center}
	where, if we view the elements of $H^0(2\kappa_C)$ as effective representatives of $2\kappa_C$, the projection $p\colon V_C \twoheadrightarrow H^0(2\kappa_C)$ maps a quadric $Q$ to
	\[
	p(Q) =
	\begin{cases}
	\mathcal{Z}(Q) \cap C &\text{if } \dim(\mathcal{Z}(Q) \cap C) = 0, \\
	0 &\text{otherwise}.
	\end{cases}
	\]
	We have that $H^0(\kappa_C + \alpha)$ is canonically isomorphic to $H^0(\mathcal O_{|\kappa_C+\alpha|^\ast}(1))$ by definition. Thus, we define $i$ to be the composition
	\[
	i\colon H^0(\mathcal O_{|\kappa_C+\alpha|^\ast}(2)) \cong \Sym^2 H^0(\kappa_C + \alpha) \longrightarrow H^0(2(\kappa_C + \alpha)) \cong H^0(2\kappa_C)
	\]
	of canonically determined morphisms.
\end{definition}

The results below are helpful for understanding Lehavi's theorem and its simplified proof.

\begin{lemma} \label{lem:conics-in-Steiner-subspace}
	For all $\alpha\in \JC[2]\setminus\{0\}$ and all $\{\theta,\theta+\alpha\}\in\Sigma_{C,\alpha}$ we have $q_{\theta, \theta+\alpha} \in V_{C,\alpha}$. Additionally, $q_C$ is trivially contained in $V_{C,\alpha}$.
\end{lemma}
\begin{proof}
	We may write the $2$-torsion class $\alpha$ as $\frac{1}{2}(\mathcal Z(\ell_\theta) \cap C - \mathcal Z(\ell_{\theta+\alpha}) \cap C)$. It is then clear that both $\ell_{\theta}$ and $\ell_{\theta+\alpha}$ cut out effective representatives of $\kappa_C + \alpha$. Since the sum of two effective representatives of $\kappa_C + \alpha$ is certainly contained in the image of $i$, we are done with the first claim. The second claim is trivial.
\end{proof}

One of the pillars in Lehavi's arguments is following semicontinuity statement \cite[Corollary 7]{Lehavi15}. It can be derived from the classical semicontinuity theorem~\cite[Section III.12]{Hartshorne77} and holds for any characteristic of~$k$.
\begin{lemma}\label{lem:lowerSemicontinuity}
	Let $\mathcal{V}/X$ be a vector bundle over a base $X$ and let $\mathcal{V}_1, \ldots, \mathcal{V}_n$ be sub-bundles of $\mathcal{V}$. Then
	\begin{itemize}
		\item[(i)] the function $\dim \langle \mathcal{V}_1 |_x , \ldots, \mathcal{V}_n |_x \rangle$ is lower semi-continuous on $X$,
		\item[(ii)] the function $\dim(\cap_{i=1}^n \mathcal{V}_i |_x )$ is upper semi-continuous on $X$.
	\end{itemize}
\end{lemma}

We will now state a generalization of Lehavi's theorem for reconstructing the quadric for complex space sextics on smooth quadrics \cite[Theorem 1]{Lehavi15}. The idea of the proof is remains the same.

\begin{theorem} \label{thm:Lehavi1}
	Let $k$ be either of characteristic zero or of sufficiently high characteristic.
	Let $C$ be a generic space sextic or a generic space sextic lying on a singular quadric over $k$. Then we have
	\begin{enumerate}[label=(\roman*)]
		\setlength\itemsep{0.2cm}
		\item
		$\PP V_{C,\alpha} = \Span(\{q_{\theta,\theta+\alpha} : \theta\in\Sigma_{C,\alpha}\}) \quad \text{for all } \alpha\in \JC[2]\setminus\{0\}$,
		\item
		$\bigcap_{\alpha\in \JC[2]\setminus\{0\}} \PP V_{C,\alpha} = \Span\{ q_C \}$,
	\end{enumerate}
	where $q_C$ denotes the unique quadric vanishing on $C$, up to scaling.
\end{theorem}
\begin{proof}
	We follow the strategy Lehavi uses for proving \cite[Theorem 1]{Lehavi15}. From Lemma~\ref{lem:conics-in-Steiner-subspace} we immediately have the containments
	$$\Span \{q_{\theta, \theta+\alpha} \} \subseteq V_{C,\alpha} \text{\quad and \quad} \Span\{q_C\} \subseteq \textstyle\bigcap_{\alpha \in \JC[2] \backslash \{0\} } V_{C,\alpha}.$$
	For claim (i), each $q_{\theta, \theta+\alpha}$ generates a rank $1$ sub-bundle of $V_{C,\alpha}$ over the family of smooth genus $4$ curves $C$ in $\PP^3$. The dimension of their span is bounded by the dimension of $V_{C,\alpha}$. If we can demonstrate that $\dim \Span{ \{q_{\theta, \theta+\alpha}\}} = \dim V_{C,\alpha}$ for a particular curve $C$, then via Lemma~\ref{lem:lowerSemicontinuity}(i) the claim must be true for all curves in an open neighbourhood of $C$.
	In the proof of \cite[Theorem 1]{Lehavi15} this is done by hand with a pleasant but involved geometric argument, whereas we can simply use our implementation of Algorithm \ref{algo: Lehavi1} as described in Example~\ref{ex:delPezzoSteinerSubspace}. In both cases the curve $C$ is over $k=\overline\QQ$, which shows the statement for both characteristic zero and sufficiently high characteristic (see the Spreading Out Theorem \cite[Theorem 3.2.1]{Poonen2017}). Additionally, we consider a space sextic on a singular quadric in Example~\ref{ex:delPezzoSteinerSubspace}, which shows the statement for generic curves of that type.

	The proof of the part (ii) is analogous, using Lemma~\ref{lem:lowerSemicontinuity}(ii) instead.
\end{proof}

\begin{algorithm} \label{algo: Lehavi1} \
	\begin{algorithmic}[1]
		\REQUIRE{$\mathcal{S}_C$, the Steiner system of $C$ as in Definition~\ref{def:SteinerSystem}.}
		\ENSURE{$q_C$, the unique quadric form vanishing on $C$.}
		\STATE Construct the subspaces
		\[ \mathbb P V_{C,\alpha} := \Span\Big(\big\{ l_\theta\cdot l_{\theta+\alpha} : \{ l_\theta,l_{\theta+\alpha}\} \in \mathcal{S}_{C,\alpha} \big\}\Big)  \text{ for all } \mathcal{S}_{C,\alpha} \in \mathcal{S}_C \vspace{-5mm}\]
		\STATE Compute their one-dimensional intersection
		\[ \Span \{ q_C \} := \bigcap_{\mathcal{S}_{C,\alpha} \in \mathcal{S}_C} \mathbb P V_{C,\alpha}. \vspace{-5mm} \]
		\RETURN{$q_C$.}
	\end{algorithmic}
\end{algorithm}

\pagebreak
\begin{example}\label{ex:delPezzoSteinerSubspace}
	Consider again the curve $C\subseteq \PP\FF_{97}^3$ from Example~\ref{ex:delPezzoCurve} and the four syzygetic tritangents $l_{12},l_{34},l_{125},l_{345}$ from Example~\ref{ex:delPezzoTritangents}. Then there exists exactly one $\mathcal{S}_{C,\alpha}\in \mathcal{S}_C$ such that $\{l_{12},l_{34}\}\in \mathcal{S}_{C,\alpha}$, for which we have $\mathbb P V_{C,\alpha}\subseteq \PP V_C$ given by the linear span
	\[ \mathbb P V_{C,\alpha}\!=\!
	\text{Lin}\!\left\{
	\begin{array}{ll}
	x_0^2 + 36x_1x_3, & x_1x_2 + 84x_1x_3+ 32x_2x_3 +68x_3^2, \\
	x_0x_1 + 85x_1x_3+ 17x_2x_3 +91x_3^2, & x_2^2 + 10x_1x_3+ 94x_2x_3 +27x_3^2, \\
	x_1^2 + 69x_1x_3+ 22x_2x_3 +45x_3^2, & x_0x_3 + 39x_1x_3+ 34x_2x_3 +38x_3^2\\
	x_0x_2 + 69x_1x_3+ 22x_2x_3 +45x_3^2,
	\end{array}
	\right\}.\]
	One can easily see that $x_1^2-x_0x_2 \in \mathbb P V_{C,\alpha}$; it is the difference of two of our generators above. Our implementation verifies that Theorem~\ref{thm:Lehavi1} (i) and (ii) hold for this particular curve.
	More generally, we can do the same for the curve $C\subseteq\mathbb P\overline \QQ^3$ obtained from the configuration $\mathcal P\subseteq\PP\QQ^3$ in Example~\ref{ex:delPezzoCurve}.
\end{example}

\subsection{Reconstructing a cubic}

Reconstructing a cubic from $\mathcal{S}_C$ is a slightly more involved task. Lehavi's idea is to identify $4$ nodes of a Cayley cubic containing $C$. \cite[Theorem 2]{Lehavi15} allows us to identify the four nodes via an intersection theoretic argument. We state a corrected version here.

\begin{theorem}[{\cite[Theorem 2]{Lehavi15}}] \label{thm:Lehavi2}
	Let $\alpha \in \JC[2] \backslash \{0\}$. The four planes through each of the four triples of nodes of the Cayley cubic associated to $\alpha$ are, set theoretically, the four intersection points in $\PP V_C$ of the six dimensional projective variety
	\[
	\PP \mathcal{Z}\!\Big( \left( (V_C/V_{C,\alpha})^* \wedge (V_C/V_{C,\alpha})^* \right) \cdot S^2 |\kappa_C| \Big)
	\]
	and the $2$nd Veronese image of $|\kappa_C|$ in $\PP V_C$. The \emph{``\ $\cdot$ \ ''} in the equation above denotes tensor contraction. Moreover, each of these four points have multiplicity $4$ in the intersection.
\end{theorem}

We provide some context for the statement of Theorem \ref{thm:Lehavi2}. To do so, it is helpful to consider the set up of \cite[Lemma 14]{Lehavi15}.

\begin{definition-lemma} \label{def:lem14setup}
	For any $\alpha\in \JC[2]\setminus\{0\}$, consider the following multiplication map~$s$, as well as the restriction of the projection onto the first coordinate on its preimage of $\PP V_{C,\alpha}$:
	\begin{center}
		\begin{tikzpicture}
		\node[anchor=base west] (KCKC) at (0,0) {$|\kappa_C| \times |\kappa_C|$};
		\node[anchor=base west,xshift=15mm] (PVC) at (KCKC.base east) {$\mathbb P V_C$};
		\draw[->] (KCKC) -- node[above] {$s$} (PVC);

		\node[anchor=north west,yshift=-5mm] (PVCa) at (PVC.south west) {$\mathbb P V_{C,\alpha}$};
		\draw[draw opacity=0] (PVCa.120) -- node[sloped] {$\subseteq$} ++(0,0.5);
		\node[anchor=north,yshift=-5mm] (kckc) at (KCKC.south) {$s^{-1} \mathbb P V_{C,\alpha}$};
		\draw[draw opacity=0] (kckc) -- node[sloped] {$\subseteq$} (KCKC);
		\draw[->] (kckc) -- (PVCa);

		\node[anchor=base east] (eq) at (kckc.base west) {$:=$};
		\node[anchor=base east] (Xa) at (eq.base west) {$X_\alpha$};
		\node[anchor=north,yshift=-7.5mm] (KC) at (kckc.south) {$|\kappa_C|$};
		\draw[->>] (kckc) -- (KC);
		\end{tikzpicture}
	\end{center}
	The two-dimensional fibers of the restricted projection lie over four points which represent four planes $l_{1,\alpha},\ldots,l_{4,\alpha}$ intersecting triplewise in four points $z_{1,\alpha},\ldots,z_{4,\alpha}$. We define the \emph{Cayley subspace} associated to $\alpha$ by
	\[
	\mathbb P W_{C,\alpha} := \Span\left\{ \frac{\prod_{j=1}^4 l_{j,\alpha} } {l_{i,\alpha}} : 1 \leq i \leq 4 \right\}
	\oplus \Span\{ h q_C : \deg(h) = 1 \}.
	\]
	Generically, the first summand is the space of cubics with the nodes at $z_{1,\alpha},\ldots,z_{4,\alpha}$ and the first summand contains the Cayley cubic associated to $\alpha$.
\end{definition-lemma}

% Operator argument symbol
\newcommand{\longunderscore}{\text{--}}
With the notation of Definition-Lemma \ref{def:lem14setup}, the planes through $3$ nodes of the Cayley cubic are exactly the $a \in |\kappa_C|$ such that the subvariety
\[
\{ x \in |\kappa_C| : v \cdot s(a,x) = 0  \text{ for all } v \in (V_C/V_{C,\alpha})^* \}
\]
has dimension $2$. For any $v \in (V_C/V_{C,\alpha})^*$, we now view $v \cdot s(a,\longunderscore) \in |\kappa_C|^*$ as an operator. As the null-spaces of these operators are identical, they are linearly dependent as vectors in $|\kappa_C|^*$. In particular, we have for any $v,v' \in (V_C/V_{C,\alpha})^*$ that
\begin{equation} \label{eq:wedge-form}
	v \cdot s(a,\longunderscore ) \wedge v' \cdot s(a,\longunderscore) = \vec 0. \tag{$\ast\ast$}
\end{equation}
In terms of a basis $\{e_1, \ldots, e_4 \}$ for $|\kappa_C|$, we have that
\[
(v \cdot s(a,e_1), \ldots, v \cdot s(a,e_4)) \wedge (v' \cdot s(a,e_1), \ldots, v' \cdot s(a,e_4)) = \vec 0.
\]
Or even more explicitly, that the $2 \times 2$ minors of the matrix
\scalebox{0.6}{
	\begin{math}
		\begin{bmatrix}
			v \cdot s(a,e_1) & \ldots & v \cdot s(a,e_4) \\
			v' \cdot s(a,e_1) & \ldots & v' \cdot s(a,e_4)
		\end{bmatrix}
	\end{math}
}
vanish. The variety defined in Theorem \ref{thm:Lehavi2} is
\[
\mathcal{Z}\!\!\left( \mathrm{minors}_{2 \times 2}
\begin{bmatrix}
v \cdot s(a,e_1) & \ldots & v \cdot s(a,e_4) \\
v' \cdot s(b,e_1) & \ldots & v' \cdot s(b,e_4)
\end{bmatrix}
: v,v' \in (V_C/V_{C,\alpha})^*
\right)
\subseteq \PP V_C
\]
where $\PP V_C = S^2|\kappa_C|$ is endowed with tensor coordinates $a \otimes b$. A coordinate-free description is given by
\[
\mathcal{Z} \big( (V_C/V_{C,\alpha})^* \wedge (V_C/V_{C,\alpha})^* \cdot S^2|\kappa_C| \big) \subseteq V_C
\]
which follows from Equation $\eqref{eq:wedge-form}$ since tensor contraction commutes with wedges.

\bigskip

%
% Space of cubic newcommand
\newcommand{\cubicsOfC}{\mathrm{I}(C)_3}
If $C = \mathcal{Z}(q_C, p_C)$ is a smooth space sextic, we have that the space $\cubicsOfC$ of cubics containing $C$ is $5$-dimensional. It is spanned by the $4$-dimensional subspace $\Span\{ h q_C : \deg(h) = 1 \}$ along with some irreducible cubic vanishing on $C$. Observe that the defining equation of any Cayley cubic containing $C$ could serve as the additional cubic. Theorem \ref{thm:Lehavi2} allows us to identify the four nodes of a Cayley cubic containing $C$ but this is not quite enough information to recover the curve uniquely. However, a Cayley cubic containing $C$ is (generically) contained in the space $\PP W_{C,\alpha}$, so by using all $\alpha \in \JC[2] \backslash \{0\}$ we can recover the space of cubics containing $C$.

\pagebreak
\begin{theorem} \label{thm:betterLehavi2}
	Let $C$ be a smooth space sextic, let $\PP \cubicsOfC$ be the $4$-dimensional subspace of cubics in $\PP^3$ containing $C$, and let $W_{C,\alpha}$ be as in Definition \ref{def:lem14setup}. Then
	\begin{enumerate}[label=(\roman*),leftmargin=*]
		\setlength\itemsep{0.2cm}
		\item
		If $C$ is generic, we have $\,\dim W_{C,\alpha} = 8 \,$ for every $\alpha \in \JC[2] \backslash \{0\}$.
		\item
		If $C$ is a generic member of the family of space sextics contained in a singular quadric, then $\, \dim W_{C,\alpha} = 8 \,$ for the $\alpha \in \JC[2] \backslash \{0\}$ not of the form $\theta - \theta_0$, where $\theta$ is one of the $120$ odd theta characteristics and $\theta_0$ is the vanishing even theta characteristic.
		\item
		Let $C$ be as in case (i) or (ii) and let $A := \{\alpha \in \JC[2] \backslash \{0\} : \dim W_{C,\alpha} = 8 \}$. Then
		\[
		\PP \cubicsOfC = \bigcap_{\alpha\in A} \PP W_{C,\alpha}.
		\]
	\end{enumerate}
\end{theorem}

The proof of Theorem \ref{thm:betterLehavi2} is similar to our proof of Theorem \ref{thm:Lehavi1}. We apply our implementation of Algorithm \ref{algo: Lehavi2} to a specific example and, via semi-continuity, we obtain that the formula holds generically.

\begin{algorithm} \label{algo: Lehavi2}\
	\begin{algorithmic}[1]
		\REQUIRE{$(\mathcal{S}_C,q_C)$, where
			\begin{itemize}[leftmargin=*]
				\item $\mathcal{S}_C$, the Steiner system as in Definition~\ref{def:SteinerSystem},
				\item $q_C$, the unique quadric form vanishing on $C$.
			\end{itemize}
		}
		\ENSURE{$p_C$, a cubic form such that $C = \mathcal{Z}(q_C)\cap\mathcal Z(p_C)$}
		\STATE We introduce the following coordinates on the homogeneous coordinate rings:
		\begin{align*}
			A(|\kappa_C|\times|\kappa_C|) &:=k[\mathbf{x}, \mathbf{y}]:= k[x_0,x_1,x_2,x_3,y_0,y_1,y_2,y_3],\\
			A(\PP V_C) &:=k[\mathbf{z}]:= k[z_{ij}: 0\leq i\leq j\leq 3],
		\end{align*}
		so that the multiplication map $s$ induces the homomorphism
		\[ s: A(\PP V_C)\longrightarrow A(|\kappa_C|\times|\kappa_C|), \quad z_{ij}\longmapsto \begin{cases}
		x_i y_j & i=j,\\
		x_i y_j + x_jy_i & i\neq j
		\end{cases}\]
		and the projection map $p$ induces the homomorphism
		\[ p: A(|\kappa_C|\times|\kappa_C|)\longrightarrow A(|\kappa_C|), \quad x_{i}\longmapsto x_i.\vspace{-5mm} \]
		\FOR{$\mathcal{S}_{C,\alpha}\in\mathcal S_C$}
		\STATE Construct the linear subspace $\mathbb P V_{C,\alpha}$ as a subvariety of $\PP V_C$:
		\[
		\mathbb P V_{C,\alpha} := \Span\Big(\big\{ l_\theta\cdot l_{\theta+\alpha} : \{ l_\theta,l_{\theta+\alpha}\} \in \mathcal{S}_{C,\alpha} \big\}\Big) \subseteq \PP V_C\cong \PP^{\{z_{ij}\}}.
		\]
		Compute the linear generators $F_{\mathbb P V_{C,\alpha}}$ of the defining ideal of $\PP V_{C,\alpha}$.
		\STATE Set $F_{X_\alpha}:=s(F_{\mathbb P V_{C,\alpha}})\subseteq k[\mathbf{x},\mathbf{y}]$, so that $X_\alpha=\mathcal{Z}(F_{X_\alpha})$.
		\STATE Let $I_\alpha$ be the ideal generated $F_{X_\alpha}$ and the $2\times 2$ minors of its Jacobian:
		\[ J_y F = \bigg(\frac{\partial f}{\partial y_j}\bigg)_{f\in F_{X_\alpha},\; j=0,\ldots,3}. \vspace{-2mm}\]
		\STATE Compute $\mathcal{Z}(p^{-1} I_\alpha)=: \{l_{1,\alpha},l_{2,\alpha},l_{3,\alpha},l_{4,\alpha}\} \subseteq |\kappa_C|$.
		\STATE Compute the four points $\{z_{1,\alpha},z_{2,\alpha},z_{3,\alpha},z_{4,\alpha}\} \subseteq |\kappa_C|^*$ which lie on exactly $3$ of the planes in $\{l_{1,\alpha},l_{2,\alpha},l_{3,\alpha},l_{4,\alpha}\}$.
		\STATE Construct the set $\mathcal W_{C,\alpha}$ of cubics in $\mathbb P^3$ with nodes $\{z_{1,\alpha},z_{2,\alpha},z_{3,\alpha},z_{4,\alpha}\}$.
		\STATE Compute $W_{C,\alpha} := \mathcal W_{C,\alpha} \oplus \Span \{x_0 q_C, \ldots, x_3 q_C \}$.
		\ENDFOR
		\STATE Set $A := \{\alpha \in \JC[2] \backslash \{0\} : \dim W_{C,\alpha} = 8  \}$.
		\STATE Compute the intersection
		\[ \PP \cubicsOfC := \bigcap_{\alpha \in A} \PP W_{C,\alpha} \]
		and pick a $p_C \in \cubicsOfC$ such that the cubic $\mathcal{Z}(p_C)$ does not contain the quadric $\mathcal{Z}(q_C)$.
		\RETURN{$p_C$.}
	\end{algorithmic}
\end{algorithm}

\begin{example} \label{ex:delPezzoCayleySubspace}
	Consider again the curve $C$ from Example~\ref{ex:delPezzoCurve}. For its $\mathbb PV_{C,\alpha}$ in Example~\ref{ex:delPezzoSteinerSubspace}, we obtain the four points
	\[\mathcal Z(p^{-1} I_\alpha) = \left\{
	\begin{array}{ll}
	(4 \!:\! 76 \!:\! 7 \!:\! 1),
	&(a^{281863} \!:\! a^{394021} \!:\! a^{855207} \!:\! 1),\\
	(a^{736807} \!:\! a^{69925} \!:\! a^{526311} \!:\! 1),
	&(a^{873223} \!:\! a^{800485} \!:\! a^{814599} \!:\! 1)
	\end{array}
	\right\}\subseteq |\kappa_C|,
	\]
	where $a\in\FF_{97^3}$ is a generator of the multiplicative group of the finite field. The four nodes of the Cayley cubic are dual to the four points above, and are given by
	\[
	\left\{
	\begin{array}{ll}
	(a^{691050} \!:\! a^{850020} \!:\! a^{167536} \!:\! 1),
	&(a^{406794} \!:\! a^{311460} \!:\! a^{735568} \!:\! 1),\\
	(a^{214122} \!:\! a^{93444} \!:\! a^{161680} \!:\! 1),
	&(43 \!:\! 25 \!:\! 72 \!:\! 1)
	\end{array}
	\right\} \subseteq |\kappa_C|^* = \PP^3.
	\]
	Under the reverse lexicographic order on the basis of the space of cubics,
	\begin{align*}
		\{ &x_0^3,  x_0^2x_1,  x_0x_1^2,  x_1^3,  x_0^2x_2,  x_0x_1x_2,  x_1^2x_2,  x_0x_2^2,  x_1x_2^2,  x_2^3,  \\
		&x_0^2x_3,  x_0x_1x_3,  x_1^2x_3,  x_0x_2x_3,  x_1x_2x_3,  x_2^2x_3,  x_0x_3^2,  x_1x_3^2,  x_2x_3^2,  x_3^3\},
	\end{align*}
	these four nodes yield
	\[
	\PP W_{C,\alpha} = \text{Lin}
	\left\{
	\begin{footnotesize}
	\arraycolsep=3pt
	\begin{array}{rrrrrrrrrrrrrrrrrrrr}
	(1&0&0&0&0&0&0&6&29&41&38&73&0&23&58&51&68&29&51&74),\\
	(0&1&0&0&0&0&0&3&11&34&57&79&0&38&46&92&41&34&61&88),\\
	(0&0&1&0&0&0&0&47&72&48&42&39&0&72&8&32&93&48&85&12),\\
	(0&0&0&1&0&0&0&3&16&2&40&8&0&85&35&90&28&69&17&15),\\
	(0&0&0&0&1&0&0&47&72&48&42&39&0&72&8&32&93&48&85&12),\\
	(0&0&0&0&0&1&0&3&16&2&40&8&0&85&35&90&28&69&17&15),\\
	(0&0&0&0&0&0&1&96&0&0&0&0&0&0&0&0&0&0&0&0),\\
	(0&0&0&0&0&0&0&0&0&0&0&0&1&96&0&0&0&0&0&0)
	\end{array}
	\end{footnotesize}
	\right\}.
	\]
\end{example}

\begin{remark}
	For our example exhibiting Theorem \ref{thm:betterLehavi2}(ii), it turns out that \linebreak $\dim W_{C,\alpha} = 7$ for every $\alpha \in \JC[2] \backslash \{0\}$ of the form $\theta - \theta_0$, with $\theta$ one of the $120$ odd theta characteristics and $\theta_0$ the vanishing even theta characteristic. We anticipate that this is always the case.
\end{remark}

\begin{conjecture}
	Let $C$ be a smooth space sextic contained in a singular quadric and let $\alpha = \theta - \theta_0 \in \JC[2] \backslash \{0\}$, where $\theta$ is one of the $120$ odd theta characteristics and $\theta_0$ the vanishing even theta characteristic. Then $\, \dim W_{C,\alpha} = 7$.
\end{conjecture}

%%*********************************************************************************
%% SECTION
%%*********************************************************************************

\section*{Acknowledgments}
The authors would like to thank David Lehavi, Emre Sert\"oz, and Bernd Sturmfels for their valuable comments. The first author was partially supported by the Henri Lebesgue Center for Mathematics. The second and third authors would like to thank the Max-Planck-Institute MIS and the Mittag-Leffler-Institute with its programme ``Tropical Geometry, Amoebas and Polytopes'' respectively for their generous hospitality.

\renewcommand*{\bibfont}{\small}
\printbibliography

\end{document}